\documentclass[reqno,twoside]{amsart}
\usepackage{amsfonts}
\usepackage{amsmath,amssymb}
\usepackage[dvips,bottom=1.4in,right=1in,top=1in, left=1in]{geometry}
\usepackage{epic}
\usepackage{pstricks}
\usepackage{graphicx}
\usepackage{xcolor}
\usepackage{a4wide,amsmath,amssymb,latexsym,amsthm}
\usepackage{graphicx}

\newtheorem{theorem}{Theorem}[section]
\newtheorem{proposition}[theorem]{Proposition}
\newtheorem{remark}[theorem]{Remark}
\newtheorem{lemma}[theorem]{Lemma}
\newtheorem{corollary}[theorem]{Corollary} 
\newtheorem{definition}[theorem]{Definition}
\newtheorem{example}[theorem]{Example}
\numberwithin{equation}{section}

\newcommand{\R}{\mathbb R}
\newcommand{\C}{\mathbb C}

\newcommand{\be}{\begin{equation}}
\newcommand{\ee}{\end{equation}}
\newcommand{\ba}{\begin{eqnarray}}
\newcommand{\ea}{\end{eqnarray}}
\newcommand{\beq}{\begin{equation}}
\newcommand{\eeq}{\end{equation}}

\numberwithin{equation}{section}

\def\H{W_0^{s,2}(\bOm)}
\def\RR{{\mathbb{R}}}
\def\NN{{\mathbb{N}}}

\def\CC{{\mathbb{C}}}
\def\C{{\mathbb{C}}}

\def\Om{\Omega}

\def\bOm{\overline{\Om}}

\usepackage{color}
\usepackage{stmaryrd}

\DeclareMathAlphabet\gothic{U}{euf}{m}{n}

\newcommand{\gota}{\gothic{a}}
\newcommand{\gotb}{\gothic{b}}

\newcommand{\abs}[1]{\left\vert{#1}\right\vert}
\newcommand{\norm}[1]{\left\Vert#1\right\Vert}


\begin{document}

\title[Fractional Gaussian estimates]{Fractional Gaussian estimates and holomorphy of semigroups}

\author{Valentin Keyantuo}
\address{V. Keyantuo, University of Puerto Rico, Rio Piedras Campus, Department of Mathematics,
 Faculty of Natural Sciences,  17 University AVE. STE 1701  San Juan PR 00925-2537 (USA)}
 \email{valentin.keyantuo1@upr.edu}
 
 \author{Fabian Seoanes}
 \address{F. Seoanes, University of Puerto Rico, Rio Piedras Campus, Department of Mathematics,
 Faculty of Natural Sciences,  17 University AVE. STE 1701  San Juan PR 00925-2537 (USA)}
 \email{fabian.seoanes@upr.edu}

\author{Mahamadi Warma}
\address{M. Warma, University of Puerto Rico, Rio Piedras Campus, Department of Mathematics,
 Faculty of Natural Sciences,  17 University AVE. STE 1701  San Juan PR 00925-2537 (USA)}
\email{mahamadi.warma1@upr.edu, mjwarma@gmail.com }

\thanks{The work of the three authors is partially supported by the Air Force Office of Scientific Research (AFOSR) under Award NO:  FA9550-18-1-0242}

\begin{abstract}
Let $\Omega\subset\R^N$ be an arbitrary open set and denote by $(e^{-t(-\Delta)_{\RR^N}^s})_{t\ge 0}$ (where $0<s<1$) the semigroup on $L^2(\RR^N)$ generated by the fractional Laplace operator.
In the first part of the paper we show that if $T$ is a self-adjoint semigroup on $L^2(\Omega)$ satisfying a fractional Gaussian estimate in the sense that $|T(t)f|\le e^{-t(-\Delta)_{\RR^N}^s}|f|$, $0\le t \le 1$, $f\in L^2(\Omega)$, then $T$ defines a bounded holomorphic semigroup of angle $\frac{\pi}{2}$ that interpolates on $L^p(\Omega)$, $1\le p<\infty$. Using a duality argument we prove that the same result also holds on the space of continuous functions.
 In the second part, we apply the above results to realization of fractional order operators with the exterior Dirichlet  conditions.
\end{abstract}

\date{}
\maketitle

%%%%%%%%%%%%%%%%%%%%
\section{Introduction }

Let $\Omega\subset \RR^N$ ($N\ge 1$) be an arbitrary open set and let $T=(T(t))_{t\ge 0}$ be a self-adjoint semigroup on $L^2(\Om)$ with generator $A$. Then by the spectral theorem, $T$ is a bounded holomorphic semigroup of angle $\frac{\pi}{2}$, that is, $T$ can be extended holomorphically to the maximal domain $\{z\in\CC:\;\operatorname{Re}(z)>0\}$. In Section   \ref{sec-21}, we recall the precise definition of a holomorphic semigroup.

Assume that  the semigroup $T$ interpolates on $L^p(\Omega)$, $1\le p<\infty$. That is, there exists for each $p$, a strongly continuous semigroup $T_p$ on $L^p(\Omega)$ with $T_2=T$ such that $T_p(t)f=T_2(t)f$ for all $f\in L^p(\Omega)\cap L^2(\Omega)$. Using the Stein interpolation theorem, it has been shown in \cite[Theorem 1.4.2]{Davis}  that for $1<p<\infty$, the semigroup $T_p$ is holomorphic on $L^p(\Omega)$  of angle $\theta_p\ge \frac{\pi}{2}\left(1-\left|\frac 2p-1\right|\right)$. 

The case $p=1$ is more delicate and has been solved by Ouhabaz in \cite{Ouhabaz} in a general context. More precisely in \cite{Ouhabaz}, the author has shown that if $T$ has a Gaussian estimate for $0\le t\le 1$ in the sense that there are two positive constants $M$ and $b$ such that

\begin{align}\label{Gauss}
|T(t)f|\le MG(bt)| f|\;\mbox{ for }\; 0\le t\le 1 \;\mbox{ and for all }\; f\in L^2(\Omega),
\end{align}
where  $G=(G(t))_{t\ge 0}$ is the Gaussian semigroup on $L^2(\RR^N)$ (the semigroup generated by the Laplace operator $\Delta$ on $L^2(\RR^N)$), then there exists $\omega\ge 0$ such that the semigroup $(e^{-\omega t}T_p(t))_{t\ge 0}$ is bounded holomorphic of angle $\frac{\pi}{2}$ on $L^p(\Omega)$ for $1\le p<\infty$. Notice that this covers most of second order elliptic operators with the three classical boundary conditions (Dirichlet, Neumann and Robin). Of course a regularity condition can be sometimes needed for the coefficients of the operator or/and the open set $\Omega$. 

Prior to the work \cite{Ouhabaz}, when $A$ is a second order elliptic operator with smooth coefficients, using a duality argument and a result by Stewart \cite{Ste1,Ste2}, it has been shown in \cite{Ama} that $T_1$ is holomorphic on $L^1(\Omega)$ of angle $0<\theta<\frac{\pi}{2}$ if $\Omega$ is a bounded smooth domain. This result has been extended in \cite{Ard-Ho} to the case of the Laplace operator with Dirichlet boundary condition on an arbitrary open set $\Omega$. The case of the Wentzell boundary condition has been investigated in \cite{War-W}. Holomorphy on spaces of continuous functions has been also studied in \cite{Ard-Ho,Ouhabaz,War-SF1,War-SF2,War-W} and the references therein.

Throughout the paper, if $X$ and $Y$ are Banach spaces and  $S: X\to Y$ is a bounded linear operator, then we shall denote by $\|S\|_{\mathcal L(X,Y)}$ the operator norm of $S$. If $X=Y$, then we shall just write $\|S\|_{\mathcal L(X)}$. We use the notation $X\hookrightarrow Y$ to indicate that the space $X$ is continuously embedded in $Y$. For a closed linear operator $A: D(A)\subset X\longrightarrow X$ the spectrum and the resolvent set of $A$ will be denoted by $\sigma(A) $ and $\rho(A)$ respectively. For $\lambda\in\rho(A)$, the corresponding resolvent operator is denoted by $R(\lambda, A).$ 

Of concern in the present paper is the investigation of the counterpart of the  problem  studied in \cite{Ouhabaz} but in our case, we consider  fractional order operators. 
We now  describe the problem  in more details.

 For $0<s<1$ we let

\begin{align*}
\mathcal L_s^{1}(\R^N):=\left\{u:\R^N\to\R\;\mbox{ measurable},\; \int_{\R^N}\frac{|u(x)|}{(1+|x|)^{N+2s}}\;dx<\infty\right\}.
\end{align*}
For $u\in \mathcal L_s^{1}(\R^N)$ and $\varepsilon>0$ we set

\begin{align*}
(-\Delta)_\varepsilon^s u(x):= C_{N,s}\int_{\{y\in\R^N:\;|x-y|>\varepsilon\}}\frac{u(x)-u(y)}{|x-y|^{N+2s}}\;dy,\;\;x\in\R^N,
\end{align*}
where $C_{N,s}$ is a normalization constant given by
\begin{align}\label{CNs}
C_{N,s}:=\frac{s2^{2s}\Gamma\left(\frac{2s+N}{2}\right)}{\pi^{\frac
N2}\Gamma(1-s)}.
\end{align}
The {\em fractional Laplacian}  $(-\Delta)^s$ is defined by the following singular integral:
\begin{align}\label{fl_def}
(-\Delta)^su(x):=C_{N,s}\,\mbox{P.V.}\int_{\R^N}\frac{u(x)-u(y)}{|x-y|^{N+2s}}\;dy=\lim_{\varepsilon\downarrow 0}(-\Delta)_\varepsilon^s u(x),\;\;x\in\R^N,
\end{align}
provided that the limit exists. 
We notice that $\mathcal L_s^{1}(\R^N)$ is the right space for which $ v:=(-\Delta)_\varepsilon^s u$ exists for every $\varepsilon>0$, $v$ being also continuous at the continuity points of  $u$.  
For more details on the fractional Laplace operator we refer to \cite{Caf3,NPV,GW-CPDE,War} and the references therein.

 Let $(-\Delta)_{\RR^N}^s$ be the operator on  $L^2(\RR^N)$ given by

\begin{equation}\label{frac-lap}
D((-\Delta)_{\RR^N}^s):=\Big\{u\in W^{s,2}(\RR^N):\; (-\Delta)^su\in L^2(\RR^N)\Big\},\;\; (-\Delta)_{\RR^N}^su=(-\Delta)^su.
\end{equation}
Then $(-\Delta)_{\RR^N}^s$ is the selfadjoint operator on $L^2(\RR^N)$ associated with the symmetric, closed and bilinear form

\begin{align}\label{calF}
\mathcal F(u,v)=\frac{C_{N,s}}{2}\int_{\RR^N}\int_{\RR^N}\frac{(u(x)-u(y))(v(x)-v(y))}{|x-y|^{N+2s}}\;dxdy,\;\; D(\mathcal F)=W^{s,2}(\RR^N).
\end{align}

We mention that $(-\Delta)_{\RR^N}^s)$ can also be defined as follows:

\begin{align*}
(-\Delta)_{\RR^N}^su=\frac{1}{\Gamma(-s)}\int_0^\infty \Big(G(t)u-u\Big)\frac{dt}{t^{1+s}},
\end{align*}
where $\Gamma(-s)=-\frac{\Gamma(1-s)}{s}$ is the Gamma function evaluated at $-s$ and we recall $G$ is the Gaussian semigroup. 

It is well-known that the operator $-(-\Delta)_{\RR^N}^s$ generates a submarkovian (positive-preserving and $L^\infty$-contractive) semigroup $(e^{-t(-\Delta)_{\RR^N}^s})_{t\ge 0}$ on $L^2(\RR^N)$ with is ultracontractive in the sense that it maps $L^1(\RR^N)$ into $L^\infty(\RR^N)$. More precisely, there is a constant $C>0$ such that

\begin{align}\label{ultra-estima}
\|e^{-t(-\Delta)_{\RR^N}^s}\|_{\mathcal L(L^1(\RR^N),L^\infty(\RR^N))}\le Ct^{-\frac{N}{2s}},\;\;\forall\;t>0.
\end{align} 

In addition, the semigroup has a kernel $0\le P_s(t,\cdot,\cdot)\in L^\infty(\RR^N\times\RR^N)$  satisfying 

\begin{align*}
\left(e^{-t(-\Delta)_{\RR^N}^s}f\right)(x)=\int_{\RR^N}P_s(t,x,y)f(y)\;dy,
\end{align*}
for every $f\in L^2(\RR^N)$.   It has been shown in \cite{Blu,ChKu} (see also \cite{Chen2}) that the kernel $P_s$ satisfies the following estimates:

\begin{align*}
P_s(t,x,y)\simeq t^{-\frac{N}{2s}}\left(1+|x-y|t^{-\frac{1}{2s}}\right)^{-(N+2s)}\;\;\,\mbox{ for all }\; x,y\in\RR^N \;\mbox{ and }\; t>0,
\end{align*}
in the sense that there are two constants $0<C_1\le C_2$ such that for all $x,y\in\RR^N$ and $t>0$, we have

\begin{equation}\label{KER}
C_1t^{-\frac{N}{2s}}\left(1+|x-y|t^{-\frac{1}{2s}}\right)^{-(N+2s)}\le P_s(t,x,y)\le C_2t^{-\frac{N}{2s}}\left(1+|x-y|t^{-\frac{1}{2s}}\right)^{-(N+2s)}.
\end{equation}

We notice that the semigroup $(e^{-t(-\Delta)_{\RR^N}^s})_{t\ge 0}$ does not have a Gaussian estimate, that is, it does not satisfy the estimate \eqref{Gauss}. Indeed, it is well-known that the kernel $K_G$ of the Gaussian semigroup $G$ is given by

\begin{align}\label{Gauss}
K_G(t,x,y)=\frac{1}{(4\pi t)^{\frac N2}}e^{-\frac{|x-y|^2}{4t}}\;\mbox{ for all }\; x,y\in\RR^N\;\mbox{ and }\; t>0.
\end{align}
Therefore if $(e^{-t(-\Delta)_{\RR^N}^s})_{t\ge 0}$  (with $0<s<1$) has a Gaussian estimate, then there would exist a constant $C>0$ such that

\begin{align}\label{gau-com}
P_s(t,x,y)\le \frac{C}{(4\pi t)^{\frac N2}}e^{-\frac{|x-y|^2}{4bt}}\;\mbox{ for all }\; x,y\in\RR^N\;\mbox{ and }\; t>0.
\end{align}
From the estimate \eqref{KER} of $P_s(t,x,y)$, it is clear that \eqref{gau-com} cannot be true.

We note that the semigroups $(e^{-t(-\Delta)_{\RR^N}^s})_{t\ge 0}$ and $G(t)_{t\ge 0}$ are related by a subordination principle. More precisely, for every $u\in L^2(\RR^N)$, we have that

\begin{equation}\label{subor}
e^{-t(-\Delta)_{\RR^N}^s}u=\int_0^\infty f_{t,s}(\tau)G(\tau)u\;d\tau\;\;\;t>0,
\end{equation}
where the function $f_{t,s}$ (known as the stable L\'evy process) is the inverse Laplace transform of the function $e^{-t\lambda^s}$. For more details we refer to \cite{AbMi,Yos}. In particular \cite{AbMi} gives the relationship between $f_{t,s}$ and the well-known Wright functions.
 It follows from \eqref{subor} that the kernel $P_s$ can be obtained from $K_G$ as follows:

\begin{align}\label{frac-Gauss}
P_s(t,x,y)=\int_0^\infty f_{t,s}(\tau) K_G(\tau,x,y)\;d\tau,\;\;x, y\in\RR^N.
\end{align}
Therefore, one can show that the estimates \eqref{KER} can be obtained by using \eqref{Gauss} and \eqref{frac-Gauss}. Since this is not the objective of the paper, we will not go into details.

The fractional Laplace operator $-(-\Delta)_{\RR^N}^s$ is known in the literature as the generator of the so-called $s$-stable L\'evy process but the name of L\'evy has not been given to the semigroup $(e^{-t(-\Delta)_{\RR^N}^s})_{t\ge 0}$. Since $\displaystyle\lim_{s\uparrow 1^-}(-\Delta)^s=-\Delta$ in the sense that

\begin{align}\label{convergence}
\lim_{s\uparrow 1^-}\int_{\RR^N}v(-\Delta)^su\;dx=\int_{\RR^N}\nabla u\cdot\nabla v\;dx=-\int_{\RR^N}v\Delta u\;dx,
\end{align}
for every $u\in W^{2,2}(\RR^N)$ and $v\in W^{1,2}(\RR^N)$,  we shall call $(e^{-t(-\Delta)_{\RR^N}^s})_{t\ge 0}$ the {\bf fractional Gaussian semigroup}. We mention that the proof of \eqref{convergence} is where the constant $C_{N,s}$ given in \eqref{CNs} plays a crucial role.

In analogy  with  \eqref{Gauss} we introduce the following notion.

\begin{definition}
We shall say that a self-adjoint semigroup $T$ on $L^2(\Om)$ has a {\bf fractional Gaussian estimate} for $0\le t\le 1$ if there are two positive constants $M$ and $b$, and some $s\in (0,\, 1)$  such that

\begin{align}\label{Levy}
|T(t)f|\le Me^{-bt(-\Delta)_{\RR^N}^s}|f|\;\mbox{ for }\; 0\le t\le 1 \;\mbox{ and for all }\; f\in L^2(\Omega).
\end{align}
If \eqref{Levy} holds for all $t\ge 0$, then we shall simply say that $T$ has a fractional  Gaussian estimate.
\end{definition}

\begin{remark}\label{remm}
{\em We mention the following facts.
\begin{enumerate}
\item If $T$ is a submarkovian semigroup on $L^2(\Omega)$, then by  \cite[Theorem 1.4.1]{Davis},  there exists a consistent family of  semigroups $T_p(t)$ on $L^p(\Omega)$ such that $T_p(t)f=T_2(t):=T(t)f$ for $f\in L^p(\Omega)\cap L^2(\Omega)$, $1\le p\le\infty$. In addition the semigroup $T_p$ is strongly continuous in $L^p(\Omega)$ if $1\le p<\infty$.

\item Now assume that $T$ is a semigroup satisfying \eqref{Levy}. Since the semigroup $(e^{-t(-\Delta)_{\RR^N}^s})_{t\ge 0}$ is submarkovian, and hence, contractive on $L^p(\RR^N)$ for $1\le p\le\infty$, we have that if \eqref{Levy} holds, then there exists a constant $\omega\ge 0$ such that the semigroup $T$ satisfies $\|T(t)f\|_{L^p(\Omega)}\le M e^{\omega t}\|f\|_{L^p(\Omega)}$ for $f\in L^p(\Omega)\cap L^2(\Omega)$, $1\le p\le \infty$. By the Riesz-Thorin interpolation theorem \cite[Page 3]{Davis}, there exists $T_p(t)\in\mathcal L(L^p(\Omega))$ such that $T_p(t)f=T_2(t)f:=T(t)f$ for $f\in L^p(\Omega)\cap L^2(\Omega)$, $1\le p\le\infty$. We shall prove in Theorem \ref{Theo-L11} that $T_p$ is strongly continuous in $L^p(\Omega)$ for  $1\le p<\infty$. 
\end{enumerate}
}
\end{remark}

Our main result shows that if $T$ has a fractional  Gaussian estimate, then $T_p$ is bounded holomorphic of angle $\frac{\pi}{2}$ on $L^p(\Omega)$ for $1\le p<\infty$. This result will be applied to the realization of fractional order operators (such as the fractional Laplace operator) with the Dirichlet exterior condition. Using a duality argument and the above holomorphy result we show that if $\Om$ is bounded and has a Lipschitz continuous boundary, then a realization of the fractional Laplacian in $C_0(\bOm)$ with zero Dirichlet exterior condition also generates a bounded holomorphic semigroup of angle $\frac{\pi}{2}$. 

 Fractional order operators have recently emerged as a modeling alternative in various branches of science. In fact, in many situations, the fractional models reflect better the behaviour of the system both in the deterministic and stochastic contexts. 
A number of stochastic models for explaining anomalous diffusion have been
introduced in the literature; among them we mention among others the fractional Brownian motion; the continuous time random walk;  the L\'evy flights; the Schneider grey Brownian motion; and more generally, random walk models based on evolution equations of single and distributed fractional order in  space (see e.g. \cite{DS,GR,Man,Sch}).  In general, a fractional diffusion operator corresponds to a diverging jump length variance in the random walk. We refer to \cite{NPV,GW-F,GW-CPDE,Val} and the references therein for a complete analysis, the derivation and the applications of fractional order operators.

The rest of the paper is organized as follows. In the first part of Section \ref{sec-2} we give some well-known results on holomorphic semigroups and fractional order Sobolev spaces as they are needed throughout the paper. In the second part of Section \ref{sec-2} we state the main results of the paper and give some examples. The proofs of the main results are given in Section \ref{sec-3}.

\section{Preleminaries and main results}\label{sec-2}

 Throughout the paper, for  $0<\psi<\pi$, we shall denote by $\Sigma(\psi)$ the sector
\begin{align*}
\Sigma(\psi):=\{z=re^{i\alpha},\; r>0,\;\;|\alpha|<\psi\}.
\end{align*}

\subsection{Holomorphy and domination of semigroups}\label{sec-21}

In this section we recall some well-known results on holomorphic semigroups and the domination criterion of semigroups that will be used throughout the paper. Let  $T=(T(t))_{t\ge 0}$  be a strongly continuous semigroup on a Banach space $X$ with generator $A$.

\begin{definition}
%\begin{enumerate}
%\item 
The semigroup $T$ is said to be  bounded holomorphic of angle $\theta\in (0,\frac{\pi}{2}]$, if $T$ has an extension to the section $\Sigma(\theta)$ which satisfies the following conditions.

\begin{enumerate}
\item[(i)] $T(z_1+z_2)=T(z_1)T(z_2)$, $z_1,z_2\in\Sigma(\theta)$.

\item[(ii)] The mapping $z\mapsto T(z)$ is holomorphic on $\Sigma(\theta)$.

\item[(iii)] For every $0<\theta_1<\theta$,\,  $\lim_{z\to 0, z\in\Sigma(\theta_1)}T(z)f=f$ for every $f\in X$.

\item[(iv)] For each $0<\theta_1<\theta$, there exists a constant $C>0$ (depending on $\theta$) such that $\|T(z)\|_{\mathcal L(X)}\le C$ for all $z\in \Sigma(\theta_1)$.
\end{enumerate}
%\item We say that $T$ is bounded holomorphic if there exists $\theta\in (0,\frac{\pi}{2}]$ such that $T$ is bounded holomorphic of angle $\theta$.
%\end{enumerate}
\end{definition}

We note that the holomorphy condition in $(ii)$ is equivalent to weak holomorphy (see e.g. \cite[Appendix A]{ABHN}).
We have the following special case.

\begin{remark}\label{rem1}
{\em
If $X$ is a Hilbert space with scalar product $(\cdot,\cdot)_X$ and $A$ is a selfadjoint generator of a bounded semigroup $T$ on $X$,  then by the spectral theorem, we have that $(Au,u)_X\le 0$. In addition the semigroup $T$ is bounded holomorphic of angle $\frac{\pi}{2}$.
}
\end{remark}

\begin{remark}\label{rem2}
{\em
It is well-known that the holomorphy of a semigroup $T$ is directly related to spectral properties of its generator $A$. In fact one has that $T$ is bounded holomorphic with angle $\theta$  if and only if $\Sigma(\theta+\frac{\pi}{2})\subset \rho(A)$ (the resolvent set of $A$), and for each $0<\theta_1<\theta$ there exists a constant $C>0$ such that

\begin{align*}
\|(\lambda-A)^{-1}\|_{\mathcal L(X)}\le \frac{C}{|\lambda|}\;\mbox{ for all }\lambda\in \Sigma(\theta+\frac{\pi}{2}).
\end{align*}
}
\end{remark}

If one is merely interested in holomorphy without specific reference to the corresponding angle, then it suffices to establish the above estimate for $\lambda\in \Sigma(\frac{\pi}{2}).$
For more details on holomorphic semigroups we refer to the monographs \cite{ABHN,Davis,Ouhabaz-Bo} and their references.

The following result is taken from \cite[Proposition 14.2.4]{Arendt}.

\begin{proposition}  \label{propint}
Let $\Omega\subset\RR^N$ be an arbitrary open set.
Let $D\subset\C$ be open and let the mapping $\mathbb S:D\rightarrow \mathcal{L}(L^2(\Omega))$ be holomorphic such that 
\begin{align}
\sup_{z\in \mathbb K}\norm{\mathbb S(z)}_{\mathcal{L}(L^1(\Omega),L^{\infty}(\Omega))}<\infty
\end{align}
for every compact set $\mathbb K\subset D$. Then there exists a function 
  $\mathbb F:D\times\Omega\times\Omega\rightarrow \C$ satisfying

\begin{align*}
  &\mathbb F(z,\cdot,\cdot)\in L^{\infty}(\Omega\times\Omega)\;\mbox{ for all } z\in D,\\
 & (\mathbb S(z)f)(x)=\int_{\Omega} \mathbb F(z,x,y)f(y)dy\;\mbox{  for a.e. } x\in\Om \;\mbox{ and for all } f\in L^1(\Omega)\cap L^2(\Omega),\\
 & \mathbb F(\cdot,x,y):D\rightarrow\C \mbox{  is holomorphic for all } x,y\in\Omega.
\end{align*}
\end{proposition}

%% about Homogeneous spaces%%

 We recall the definition of spaces of homogeneous type. If $(E, \mu, d)$ is a metric space (with distance $d$) endowed with a measure $\mu$, then for $x\in E,\, r>0,$ we denote by $B(x, r)$ the open ball in $E$ with radius $r$ and center at $x$. We say that 
$(E,\mu, d)$ has the volume doubling property if  there exists a constant $C\ge 1$ such that for all  $x\in E,\, r>0,$ $\abs{B(x,2r)}\leq C\abs{B(x,r)}$ where we use the notation $\mu(B(x,r))=\abs{B(x,2r)}.$ With this property, $(E, \mu, d)$ ia called a space of homogeneous type. Spaces of homogeneous type are considered more generally
for quasi-metric spaces but we shall not need such generality. More information on spaces of homogenous type may be found in \cite[Proposition 3.3]{D-R} and 
\cite[Chapter 7]{Ouhabaz-Bo}.\\

The following extension result is contained in \cite[Proposition 3.3]{D-R}.

 \begin{proposition}\label{Extcomp}
  Let $(E,\mu,d)$ be a  space of homogeneous type.  Let $0<\psi\le \frac{\pi}{2}$ and let $z\in\Sigma(\psi)\mapsto \mathbb K(z,x,y)\in\mathbb{C}$ ($x,y\in E$) be the kernel of a holomorphic family of bounded operators on $L^2(X,\mu)$. Assume that $\mathbb K$ satisfies the following estimates for some $m>0$:
  \begin{enumerate}
      \item There is a constant $C_1>0$ such that 
      $$\abs{\mathbb K(z,x,y)}\leq C_1\abs{B(x,(\operatorname{Re}( z))^{\frac 1m}}^{-1}$$
      for all $x,y\in E$ and $z\in\Sigma(\psi)$.
      \item There is a bounded decreasing function $\xi:\RR\to\RR_+$ such that
      $$\abs{\mathbb K(t,x,y)}\leq \abs{B(x,t^{\frac 1m})}^{-1}\xi\left(d(x,y)^m t^{-1}\right)$$
      for all $x,y\in E$ and $t>0$.
  \end{enumerate}
Then for each $\varepsilon\in(0,1]$ and $\theta\in(0,\varepsilon\psi)$, there is a constant $C>0$ such that 
  \begin{equation}
      \abs{\mathbb K(z,x,y)}\leq C \abs{B(x,(\operatorname{Re} (z))^{\frac 1m}}^{-1}\xi\left(d(x,y)^m \abs{z}^{-1}\right)^{1-\varepsilon}
  \end{equation}
for all  $x,y\in E$ and $z\in\Sigma(\theta)$. 
  \end{proposition}

We conclude this section by stating  the following  result concerning domination of semigroups, which is  taken from \cite{Ouh}.

\begin{proposition}\label{prop-Ouh}
Let $(\gota,D(\gota))$ and  $(\gotb, D(\gotb))$ be two positive, symmetric and continuous bilinear forms on $L^2(\Omega)$.
Let $T$ and $S$ be the self-adjoint semigroups on $L^2(\Omega)$ associated with $\gota$ and $\gotb$, respectively. Assume that the semigroups $T$ and $S$ are positive. Then the following assertions are equivalent.

\begin{enumerate}
\item[(i)] The semigroup $T$ is dominated by the semigroup $S$ in the sense that
\begin{align*}
|T(t)f|\le S(t)|f|,\;\;\mbox{ for all }\; f\in L^2(\Omega)\;\mbox{ and }\;t\ge 0.
\end{align*}

\item[(ii)] $D(\gota)$ is an ideal in $D(\gotb)$, in the sense that if $0\le v\le u$ with $u\in D(\gota)$ and $v\in D(\gotb)$, then $v\in D(\gota)$, and 
\begin{align*}
\gotb(u,v)\le \gota(u,v)\;\mbox{ for all }\; 0\le u,v\in D(\gota).
\end{align*}
\end{enumerate}
\end{proposition}

\subsection{Fractional order Sobolev spaces}
In this section we introduce the needed fractional order Sobolev spaces. Given $0<s<1$ and $\Omega\subset\R^N$ an arbitrary open set whose closure we denote by 
$\overline{\Omega}$, we set
\begin{align*}
W^{s,2}(\Omega):=\left\{u\in L^2(\Om):\;\int_\Omega\int_\Omega\frac{|u(x)-u(y)|^2}{|x-y|^{N+2s}}\;dxdy<\infty\right\},
\end{align*}
and we endow it with the norm defined by
\begin{align*}
\|u\|_{W^{s,2}(\Omega)}:=\left(\int_\Omega|u(x)|^2\;dx+\int_\Omega\int_\Omega\frac{|u(x)-u(y)|^2}{|x-y|^{N+2s}}\;dxdy\right)^{\frac 12}.
\end{align*}
We set
\begin{align*}
W_0^{s,2}(\bOm):=\Big\{u\in W^{s,2}(\R^N):\;u=0\;\mbox{ in }\;\R^N\setminus \Omega\Big\}.
\end{align*}

We have the following continuous embeddings (see e.g. \cite{NPV}):
\begin{itemize}

\item Let $2^\star:= \frac{2N}{N-2s}$ if $N>2s$ and $2^\star\in [2,\infty)$  be arbitrary  if $N=2s$. Then
\begin{align}\label{sob-imb}
W_0^{s,2}(\bOm)\hookrightarrow L^{2^\star}(\Omega).
\end{align}
\item If $N<2s$, then
\begin{align*}
W_0^{s,2}(\bOm)\hookrightarrow C_c^{0,s-\frac N2}(\R^N).
\end{align*}
\end{itemize}

For more information on fractional order Sobolev spaces, we refer to \cite{NPV,Gris,War} and the corresponding  references.

\subsection{Main results and examples} In this section we state the main results of the article and give some examples. Recall that $\Omega\subset\RR^N$ is an arbitrary open set and $T=(T(t))_{t\ge 0}$ is a self-adjoint semigroup on $L^2(\Omega)$ with generator $A$. Recall also that we say that $T$ has a fractional Gaussian estimate for $0\le t\le 1$ if \eqref{Levy} holds and that $T$ has a fractional Gaussian estimate  if  \eqref{Levy} holds for all $t\ge 0$.

The following theorem is the first main result of the paper.

\begin{theorem}\label{Theo-L11}
Assume that $T$ has a fractional Gaussian estimate for $0\le t\le 1$. Then there is a constant $\omega\ge 0$ such that the semigroup $(e^{-\omega t}T_p(t))_{t\ge 0}$ is bounded holomorphic of angle $\frac{\pi}{2}$ on $L^p(\Omega)$ for every $1\le p<\infty$.
\end{theorem}

Next we give an example.

\begin{example}\label{ex1}
Let $\Omega\subset\RR^N$ be an arbitrary opens set.
Let $\mathcal E:D(\mathcal E)\times D(\mathcal E)\to\RR$ with $D(\mathcal E)=\H$ be the closed, continuous, non-negative and symmetric bilinear form defined by

\begin{align}\label{matE}
\mathcal E(u,v):=\frac{C_{N,s}}{2}\int_{\R^N}\int_{\R^N}\frac{(u(x)-u(y))(v(x)-v(y))}{|x-y|^{N+2s}}\;dxdy.
\end{align}
Let $(-\Delta)_D^s$ be the self-adjoint operator on $L^2(\Omega)$ associated with $\mathcal E$ in the sense that

\begin{equation}\label{opH}
\begin{cases}
D((-\Delta)_D^s):=\Big\{u\in \H:\;\exists\;f\in L^2(\Omega);\; \mathcal E(u,v)=(f,v)_{L^2(\Omega)},\;\;\forall\;v\in\H\Big\},\\
(-\Delta)_D^su=f.
\end{cases}
\end{equation}
Then $(-\Delta)_D^s$ is the realization in $L^2(\Om)$ of $(-\Delta)^s$ with the zero Dirichlet exterior condition: $u=0$  in $\R^N\setminus\Omega$ and is given precisely by

\begin{equation}\label{opH2}
D((-\Delta)_D^s)=\Big\{u\in\H:\; (-\Delta)^su\in L^2(\Omega)\Big\},\;(-\Delta)_D^su=(-\Delta)^su.
\end{equation}
We have the following results.

\begin{enumerate}
\item The operator $-(-\Delta)_D^s$ generates a strongly continuous submarkovian semigroup $(T_s(t))_{t\ge 0}$ on $L^2(\Omega)$. 

\item The semigroup $T_s$ has a fractional Gaussian estimate.
\end{enumerate}
As a consequence, it follows from Theorem \ref{Theo-L11} that the semigroup $(T_{s,p}(t))_{t\ge 0}$ is bounded holomorphic of angle $\frac{\pi}{2}$ on $L^p(\Omega)$ for every $1\le p<\infty$.

\begin{proof}
 (a) Firstly, since the symmetric closed form $\mathcal E$ is non-negative, elliptic, continuous and $\H$ is dense in $L^2(\Om)$, we have that the operator $-(-\Delta)_D^s$ generates a strongly continuous semigroup $T_s=(e^{-t(-\Delta)_D^s})_{t\ge 0}$ on $L^2(\Omega)$. 

Secondly, we claim that the semigroup is positive-preserving. Indeed, let $u\in\H$.  We consider the decomposition $u=u^+-u^-$ of $u$ where $u^+$ and $u^-$ are the positive and negative parts of $u$, respectively. Then by \cite[Lemma 2.6]{War}, $u^+\in\H$. Since
\begin{align*}
\Big(u^+(x)-u^+(y)\Big)\Big(u^-(x)-u^-(y)\Big)=&u^+(x)u^-(x)+ u^+(y)u^-(y)-u^+(x)u^-(y)-u^+(y)u^-(x)\\
=&-\Big(u^+(x)u^-(y)+u^+(y)u^-(x)\Big)\le 0,
\end{align*}
for a.e. $x,y\in\Omega$, it follows that

\begin{align}\label{For-Ne}
\mathcal E(u^+,u^-)=&\frac{C_{N,s}}{2}\int_{\R^N}\int_{\R^N}\frac{(u^+(x)-u^+(y))(u^-(x)-u^-(y))}{|x-y|^{N+2s}}\;dxdy\notag\\
=&-\frac{C_{N,s}}{2}\int_{\R^N}\int_{\R^N}\frac{u^+(x)u^-(y)+u^+(y)u^-(x)}{|x-y|^{N+2s}}\;dxdy\le 0.
\end{align}
By \cite[Theorem 1.3.2]{Davis}, the inequality  \eqref{For-Ne} implies that the semigroup $T_s$ is positive-preserving.

Thirdly, we claim that $T_s$ is $L^\infty$-contractive. Let $0\le u\in \H$. It follows from \cite[Lemma 2.7]{War} that $u\wedge 1\in \H$. 
We shall show that 
\begin{align*}
\mathcal E(u\wedge 1,u\wedge 1)\le \mathcal E(u,u) \;\mbox{  for  every \ } u\in \H,\, u\ge 0.
\end{align*}
Let
\begin{align*}
A:=\{x\in\R^N:\;u(x)\le 1\}\;\mbox{ and }\; B:=\R^N\setminus A.
\end{align*}
For almost every $(x,y)\in A\times A$, we have that
\begin{align}\label{A1}
\Big((u\wedge 1)(x)-(u\wedge 1)(y)\Big)^2= \Big(u(x)\wedge 1-u(y)\wedge 1\Big)^2= (u(x)-u(y))^2.
\end{align}
For almost every $(x,y)\in A\times B$, we have that
\begin{align}\label{A2}
\Big((u\wedge 1)(x)-(u\wedge 1)(y)\Big)^2=\Big(u(x)\wedge 1-u(y)\wedge 1\Big)^2=(u(x)-1)^2 \le (u(y)-u(x))^2.
\end{align}
For almost every $(x,y)\in B\times B$, we have that
\begin{align}\label{A3}
\Big((u\wedge 1)(x)-(u\wedge 1)(y)\Big)^2= 0.
\end{align}
It follows from \eqref{A1}, \eqref{A2} and \eqref{A3} that

\begin{align}\label{ine-sub}
\mathcal E(u\wedge 1,u\wedge 1)=&\frac{C_{N,s}}{2}\int_{\R^N}\int_{\R^N}\frac{((u\wedge 1)(x)-(u\wedge 1)(y))((u\wedge 1)(x)-(u\wedge 1)(y))}{|x-y|^{N+2s}}\;dxdy\notag\\
=&\frac{C_{N,s}}{2}\int_{}\int_{A}\frac{((u\wedge 1)(x)-(u\wedge 1)(y))((u\wedge 1)(x)-(u\wedge 1)(y))}{|x-y|^{N+2s}}\;dxdy\notag\\
&+2\frac{C_{N,s}}{2}\int_{A}\int_{B}\frac{((u\wedge 1)(x)-(u\wedge 1)(y))((u\wedge 1)(x)-(u\wedge 1)(y))}{|x-y|^{N+2s}}\;dxdy\notag\\
&+\frac{C_{N,s}}{2}\int_{B}\int_{B}\frac{((u\wedge 1)(x)-(u\wedge 1)(y))((u\wedge 1)(x)-(u\wedge 1)(y))}{|x-y|^{N+2s}}\;dxdy\notag\\
\le &\frac{C_{N,s}}{2}\int_{\R^N}\int_{\R^N}\frac{|u(x)-u(y)|^2}{|x-y|^{N+2s}}\;dxdy=\mathcal E(u,u).
\end{align}
By  \cite[Theorem 1.3.3]{Davis}, the estimate \eqref{ine-sub} implies that the semigroup $T_s$ is $L^\infty$-contractive.
We have shown that the semigroup $T_s$ is submarkovian and the proof of Part (a) is complete.

(b) Recall that the semigroups $T_s$ and $(e^{-t(-\Delta)_{\RR^N}^s})_{t\ge 0}$ are positive.
Let $0\le u\le v$ with $u\in W^{s,2}(\RR^N)$ and $v\in W_0^{s,2}(\bOm)$. Then it is clear that $u=0$ in $\RR^N\setminus\Omega$. Hence, $u\in W_0^{s,2}(\bOm)$. We have shown that $W_0^{s,2}(\bOm)$ is an ideal in $W_0^{s,2}(\RR^N)$. Let $0\le u,v\in W_0^{s,2}(\bOm)$. Clearly we have that

\begin{align*}
\mathcal E(u,v)= \mathcal F(u,v),
\end{align*}
where we recall that  $\mathcal F$ has been defined in \eqref{calF}.
By Proposition \ref{prop-Ouh} the above properties imply that  the semigroup $(T_s(t))_{t\ge 0}$ is dominated by the semigroup $(e^{-t(-\Delta)_{\RR^N}^s})_{t\ge 0}$, that is, $T_s$ has a fractional Gaussian estimate. The proof is finished
\end{proof}
\end{example}

\begin{remark}
{\em Firstly, it follows from the embedding \eqref{sob-imb} that if $\Omega$ is bounded, then the embedding $W_0^{s,2}(\bOm)\hookrightarrow L^2(\Omega)$ is compact and this would imply that the operator
$(-\Delta)_D^s$ has a compact resolvent and its spectrum consists of   eigenvalues $(\lambda_n)_{n\in\NN}$ satisfying $0<\lambda_1\le\lambda_2\le\cdots\le\lambda_n\le\cdots$ and $\lim_{n\to\infty}\lambda_n=\infty$.
Secondly, let us mention that the operator $(-\Delta)_D^s$ is different from the spectral Dirichlet fractional Laplacian, that is, the fractional $s$-power of the Laplace operator with zero Dirichlet boundary condition. Their first eigenvalues are different, the eigenfunctions of the spectral Dirichlet fractional Laplacian are smooth (it has the same eignefunctions as the Dirichlet Laplacian) and this is not the case for $(-\Delta)_D^s$ where the eigenfunctions are even not Lipschitz continuous. For more details we refer to \cite{BWZ1,SV2}.
}
\end{remark}

The following result which can be viewed as a corollary of Theorem \ref{Theo-L11} is our second main result.

\begin{corollary}\label{Theo-C}
Let $C_0(\Om):=\{u\in C(\Omega):\; \lim_{|x|\to\infty}u(x)=0\}$. Assume that $T$ has a fractional Gaussian estimate for $0\le t\le 1$ and there exists a semigroup $T_0$ on $C_0(\Omega)$ such that $T_0(t)f=T(t)f$ for $f\in C_0(\Omega)\cap L^2(\Omega)$. Then there exists $\omega\ge 0$ such that the semigroup $(e^{-\omega t}T_0(t))_{t\ge 0}$ is bounded holomorphic of angle $\frac{\pi}{2}$ on $C_0(\Omega)$. If $\Omega$ is bounded then the same conclusion also holds if one replaces $C_0(\Om)$ by $C(\bOm)$.
\end{corollary}

%\begin{theorem}%\label{Theo-C}
%Assume that $\Omega\subset\R^N$ is a bounded open set with a Lipschitz continuous boundary. Let $(-\Delta)_{D}^s$ be the operator defined in \eqref{opH}-\eqref{opH2}.
%Let $(-\Delta)_{D,c}^s$ be the part of $(-\Delta)_D^s$ in $C_0(\bOm):=\{u\in C_c(\R^N),\;\;u=0\mbox{ in }\; \R^N\setminus\Omega\}$. Then $-(-\Delta)_{D,c}^s$ generates a  bounded holomorphic semigroup of angle $\frac{\pi}{2}$on $C_0(\bOm)$.
%\end{theorem}

Next we give an example.

\begin{example}
Assume that $\Omega\subset\R^N$ is a bounded open set with a Lipschitz continuous boundary. Let $(-\Delta)_{D}^s$ be the operator defined in \eqref{opH}-\eqref{opH2}.
Let $(-\Delta)_{D,c}^s$ be the part of $(-\Delta)_D^s$ in $C_0(\bOm):=\{u\in C_c(\R^N),\;\;u=0\mbox{ in }\; \R^N\setminus\Omega\}$. Then $-(-\Delta)_{D,c}^s$ generates a  bounded holomorphic semigroup of angle $\frac{\pi}{2}$  on $C_0(\bOm)$.

\begin{proof}
We prove the result in several steps.

{\bf Step 1}.
Let $\lambda>0$ and consider the following exterior Dirichlet problem:

\begin{equation}\label{DP}
\begin{cases}
(-\Delta)^su +\lambda u=f\;\;&\mbox{ in }\;\Omega,\\
u=0&\mbox{ in }\;\RR^N\setminus\Omega.
\end{cases}
\end{equation}

By a weak solution of \eqref{DP} we mean a function $u\in W_0^{s,2}(\bOm)$ such that
\begin{align*}
\mathcal E(u,v)+\lambda \int_{\Omega}uv\;dx=\int_{\Om}fv\;dx,\;\;\forall \;v\in W_0^{s,2}(\bOm),
\end{align*}
provided that the right hand side makes sense and where $\mathcal E$ is given in \eqref{matE}.

It has been shown in \cite{RS2-2} that if $f\in L^\infty(\Omega)$, then the Dirichlet problem \eqref{DP} has a unique weak solution $u$; moreover, $u\in W_0^{s,2}(\bOm)\cap C_0(\bOm)$.\\

{\bf Step 2}. Let $\lambda>0$ be given  and denote by $R(\lambda,(-\Delta)_D^s)$ the resolvent of $(-\Delta)_D^s$. That is, the operator $(\lambda+(-\Delta)_D^s)^{-1}$.
We claim that $R(\lambda,(-\Delta)_D^s)(C_0(\bOm))\subset C_0(\bOm)$ and is dense in $C_0(\bOm)$. Let $f\in L^\infty(\Omega)$ and set $u:=R(\lambda,(-\Delta)_D^s)f$. Then $u\in W_0^{s,2}(\bOm)$ and satisfies the equality
\begin{align}
\mathcal E(u,v)+\lambda \int_{\Omega}uv\;dx=\int_{\Omega}fv\;dx,\;\;\; \forall \;v\in W_0^{s,2}(\bOm).
\end{align}
Hence, $u$ is a weak solution of the Dirichlet problem \eqref{DP}. It follows from Step 1 that $u\in W_0^{s,2}(\bOm)\cap C_0(\bOm)$. Thus, $R(\lambda,(-\Delta)_D^s)(L^\infty(\Omega))\subset C_0(\bOm)$ and in particular, we have that $R(\lambda,(-\Delta)_D^s)(C_0(\bOm))\subset C_0(\bOm)$.  Denote by $\mathcal D(\Omega)$ the space of test functions on $\Omega.$ Since $\mathcal D(\Omega)\subset R(\lambda,(-\Delta)_D^s)(C_0(\bOm))$ and is dense in $C_0(\bOm)$, it follows that $R(\lambda,(-\Delta)_D^s)(C_0(\bOm))$ is dense in $C_0(\bOm)$ and the claim is proved.  We have also shown that $T_s(t)$ leaves the space $C_0(\bOm)$ invariant, that is, $T_s(t)(C_0(\bOm))\subset C_0(\bOm)$.\\

{\bf Step 3}. Let $(-\Delta)_{D,c}^s$ be the operator defined on $C_0(\bOm)$ by

\begin{equation}
\begin{cases}
D((-\Delta)_{D,c}^s)=\Big\{u\in D((-\Delta)_{D}^s)\cap C_0(\bOm):\; (-\Delta)_{D}^su\in C_0(\bOm)\Big\},\\
(-\Delta)_{D,c}^su=(-\Delta)_{D}^su.
\end{cases}
\end{equation}
It follows from Step 2 that $(-\Delta)_{D,c}^s$ is well-defined and is the part of $(-\Delta)_{D}^s$ in $C_0(\bOm)$. 
%More precisely, we have that
%
%
%\begin{equation*}
%\begin{cases}
%D((-\Delta)_{D,c}^s)=\Big\{u\in W_0^{s,2}(\bOm\cap C_0(\bOm):\; (-\Delta)_{D}^su\in C_0(\bOm)\Big\},\\
%(-\Delta)_{D,c}^su=(-\Delta)^su.
%\end{cases}
%\end{equation*}

Recall that the semigroup $T_s$ is submarkovian. Hence, by Remark \ref{remm} there are consistent semigroups on $L^p(\Omega)$, $1\le p\le\infty$. We denote by $-(-\Delta)_{D,\infty}^s=-[(-\Delta)_{D,1}^s]^\star$ the generator of the semigroup on $L^\infty(\Omega)$ where $-[(-\Delta)_{D,1}^s]^\star$ is the dual of the generator of the semigoup on $L^1(\Omega)$. The consistency property together with Step 2 imply that for each real number $\lambda>0$, we have 

\begin{align*}
R(\lambda, (-\Delta)_{D}^s)(L^\infty(\Omega))=R(\lambda,(-\Delta)_{D,\infty}^s)(L^\infty(\Omega))=D((-\Delta)_{D,\infty}^s)\subset C_0(\bOm).
\end{align*}
Since $C_0(\bOm)$ is a closed subspace of $L^\infty(\Omega)$, and $D((-\Delta)_{D,c}^s)=D((-\Delta)_{D,\infty}^s)$ is dense in $C_0(\bOm)$ (by Step 2), and $T_s(t)$ leaves $C_0(\bOm)$ invariant (by Step 2), it follows that the operator $-(-\Delta)_{D,c}^s$ generates a strongly continuous semigroup $T_{s,0}$ on $C_0(\bOm)$. Since $T_s$ has a fractional Gaussian estimate (by Example \ref{ex1}), then the result follows from Corollary \ref{Theo-C}. 
%it follows from \cite[Remark 3.7.13]{ABHN} that the semigroup $(e^{-t(-\Delta)_{D,c}^s})_{t\ge 0}$ is bounded holomorphic of angle $\frac{\pi}{2}$ on $C_0(\bOm)$. The proof is finished.
\end{proof}

\end{example}

\section{Proof of the main results}\label{sec-3}

%\subsection{Proof of Theorem \ref{Theo-L11}}

In this section we give the proof of the main results. In order to proceed with the proof we need  some preliminary results.  We start with the following.

\begin{lemma}\label{lem31}
Assume that $T$ has a fractional Gaussian estimate for $0\le t\le 1$. Then there exists $\omega\ge 0$ such that the semigroup $(e^{-\omega t}T(t))_{t\ge 0}$ has a fractional Gaussian estimate for all $t\ge 0$.
\end{lemma}

\begin{proof}
By assumption $|T(t)f|\le M e^{-tb(-\Delta)_{\RR^N}^s}|f|$ for $0\le t\le 1$ and $f\in L^2(\Omega)$.
Let $t\ge 1$ and write $t=n+\tau$ with $0\le \tau<1$ and $n\in\NN$. Then using the semigroup property we get that

\begin{align*}
|T(t)f|=&|T(n)T(\tau)f|=|(T(1))^nT(\tau)f|\\
\le& M^{n+1} e^{-nb(-\Delta)_{\RR^N}^s} e^{-\tau b(-\Delta)_{\RR^N}^s}|f|=M^{n+1} e^{-tb(-\Delta)_{\RR^N}^s}|f|\le M e^{\omega t} e^{-tb(-\Delta)_{\RR^N}^s}|f|,
\end{align*}
for every $f\in L^2(\Om)$, where $\omega=\ln(M)$ and we have assumed that $M\ge 1$ . The proof is finished.
\end{proof}

Recall that the  semigroup $(e^{-t(-\Delta)_{\RR^N}^s})_{t\ge 0}$ is submarkovian, hence it is contractive on $L^p(\RR^N)$ for $1\le p\le\infty$ (by Lemma \ref{lem31}). This implies that if \eqref{Levy} holds, then there exists $\omega\ge 0$ such that the semigroup $T$ satisfies $\|T(t)f\|_{L^p(\Om)}\le Me^{\omega t}\|f\|_{L^p(\Omega)}$ for every $f\in L^p(\Omega)\cap L^2(\Omega)$, $1\le p\le\infty$. By the Riesz-Thorin interpolation theorem, there exists $T_p(t)\in\mathcal L(L^p(\Omega))$ such that $T_p(t)f=T_2(t)f:=T(t)f$ for every $f\in L^p(\Omega)\cap L^2(\Omega)$, $1\le p\le\infty$. One can easily show that the semigroup $T_p$ is strongly continuous on $L^p(\Omega)$ if $1<p<\infty$.

%It is easy to see from \eqref{Levy} that $T(t)$ maps $L^2(\Omega)$ into $L^\infty(\Omega)$ for each $t>0$. This implies that $T(t)$ is given by a kernel $K(t,x,y)$. Thta is,
%
%The same is also true for $T(z)$ for $z\in\Sigma(\frac{\pi}{2})$. Denote by $T(z)$ the corresponding kernel of $T(z)$.

\begin{lemma}
Assume that $T$ has a fractional Gaussian estimate. Then the operator $T(t)$ is given by a kernel $K(t,\cdot,\cdot)\in L^\infty(\Omega\times\Omega)$. That is,

\begin{align}\label{in-ker}
(T(t)f)(x)=\int_{\Omega}K(t,x,y)f(y)\;dy\;\;\mbox{ for each }\; t>0,\;f\in L^2(\Omega)\;\mbox{ and }\; x\in\Omega.
\end{align}
The same is true for $T(z)$, $z\in\Sigma(\frac{\pi}{2})$. Denoting by $K(z,\cdot,\cdot)$ the corresponding kernel of $T(z)$, we have that
there is a constant $C>0$ such that

\begin{equation}\label{ker-com}
   \abs{ K(z,x,y)}\leq M(\operatorname{Re} (z))^{-\frac{N}{2s}}\left(1+\abs{x-y}\abs{bz}^{-1/2s}\right)^{-(N+2s)(1-\varepsilon)} 
\end{equation}
for all $x,y\in \Omega$ and $z\in\Sigma(\frac{\pi}{2})$. 
\end{lemma}

\begin{proof}
Firstly, it follows from \eqref{Levy} and \eqref{ultra-estima} that  $T(t)$ maps $L^2(\Omega)$ into $L^\infty(\Omega)$ and there is a constant $C>0$ such that

\begin{equation}\label{inq0}
\norm{T(t)}_{\mathcal{L}(L^2(\Omega),L^\infty(\Omega))}\leq Ct^{-\frac{N}{4s}},\hspace{0.5cm}t>0.
\end{equation}
By duality, we have that there is a constant $C>0$ such that

\begin{equation}\label{inq2}
\norm{T(t)}_{\mathcal{L}(L^1(\Omega),L^2(\Omega))}\leq Ct^{-\frac{N}{4s}},\hspace{0.5cm}t>0.
\end{equation}
It follows from \eqref{inq0} and \eqref{inq2} that $T(t)$ maps $L^1(\Omega)$ into $L^\infty(\Omega)$ and there is a constant $C>0$ such that

\begin{equation}\label{inq00}
\norm{T(t)}_{\mathcal{L}(L^1(\Omega),L^\infty(\Omega))}\leq Ct^{-\frac{N}{2s}},\hspace{0.5cm}t>0.
\end{equation}
By Proposition \ref{propint} this shows that $T(t)$ is given by a kernel $K(t,\cdot,\cdot)\in L^\infty(\Omega\times\Omega)$ satisfying \eqref{in-ker}.

Secondly, since the  semigroup $T$ is bounded holomorphic of angle $\frac{\pi}{2}$ on $L^2(\Omega)$, it follows that for every $0<\theta<\frac{\pi}{2}$ there exist $C\geq0$ and $w\in\R$ such that 
\begin{align*}
\norm{T(z)}_{\mathcal{L}(L^2(\Omega))}\leq Me^{\abs{z}w},\hspace{0.5cm}z\in\Sigma(\theta).
\end{align*}
Let $0<\theta_2<\theta_3<\frac{\pi}{2}$. Replacing $(T(t))_{t\geq0}$ by $(e^{-w t}T(t))_{t\geq0}$ if necessary, we may assume that 

\begin{equation}\label{inq1}
\norm{T(z)}_{\mathcal{L}(L^2(\Omega))}\leq C,\hspace{0.5cm}z\in\Sigma(\theta_3).
\end{equation}
Next, let $\delta\in (0,1)$ be such that  $\delta t+is\in\Sigma(\theta_2)$ whenever $\delta t+is\in\Sigma(\theta_3)$. Let $z=t+is\in\Sigma(\theta_2)$. Then 

\begin{align*}
z=t+is=\frac{(1-\delta)t}{2}+(\delta t+is) +\frac{(1-\delta)t}{2}.
\end{align*}
Using (\ref{inq1}), \eqref{inq0}, (\ref{inq2}) and the semigroup property, we can deduce that there is a constant $C>0$ such that
\begin{align*}
&\norm{T(z)}_{\mathcal{L}(L^1(\Om),L^{\infty}(\Omega))}\\
&\leq\norm{T\left(\frac{(1-\delta)t}{2}\right)}_{\mathcal{L}(L^1(\Om),L^{2}(\Om))}\norm{T(\delta t+is)}_{\mathcal{L}(L^{2}(\Om))}\norm{T\left(\frac{(1-\delta)t}{2}\right)}_{\mathcal{L}(L^2(\Om),L^{\infty}(\Om))}\\
&\leq Ct^{-\frac{N}{2s}}=C\Big(\operatorname{Re}( z)\Big)^{-\frac{N}{2s}}.
\end{align*}
It follows from Proposition \ref{propint} that the kernel $K:\Sigma(\theta_2)\times\Omega\times\Omega\rightarrow\C$ is holomorphic, in addition $K(z,\cdot,\cdot)\in L^{\infty}(\Omega\times\Omega)$ and

\begin{align}\label{in-rep}
\Big(T(z)f\Big)(x)=\displaystyle\int_{\Omega} K(z,x,y)f(y)dy\;\mbox{  for all }\; f\in L^p(\Omega)\cap L^2(\Omega).
\end{align}
By the Dunford-Pettis criterion, we have that there exists a constant $C>0$ such that

\begin{equation}
    \abs{K(z,x,y)}\leq C\Big(\operatorname{Re}( z)\Big)^{-\frac{N}{2s}}\;\;\;(z\in\Sigma(\theta_2)).
\end{equation}

It follows from \eqref{Levy} that there is a constant $C>0$ such that
\begin{equation}
    0\le K(t,x,y)\leq MP_s(bt,x,y) 
\leq C t^{-\frac{N}{2s}}\left(1+\abs{x-y}(bt)^{-1/2s}\right)^{-(N+2s)} \;\;\mbox{ for all }\;\;x,y\in\Omega.
\end{equation}
We have shown that the conditions of Proposition \ref{Extcomp} are satisfied. Therefore, for every $\varepsilon\in(0,1]$ and $\theta_1\in(0,\varepsilon\theta_2)$, there is a constant $C>0$ such that 

\begin{equation}\label{estma}
   \abs{ K(z,x,y)}\leq C\Big(\operatorname{Re} (z)\Big)^{-\frac{N}{2s}}\left(1+\abs{x-y}\abs{bz}^{-1/2s}\right)^{-(N+2s)(1-\varepsilon)} 
\end{equation}
for  every  $x,y\in \Omega$ and  $z\in\Sigma(\theta_1)$. Since $0<\theta_1<\frac{\pi}{2}$ was arbitrary, we have that \eqref{estma} remains true for every $z\in \Sigma(\frac{\pi}{2})$. We have shown \eqref{ker-com} and the proof is finished.
\end{proof}

 \begin{proof}[\bf Proof of Theorem \ref{Theo-L11}]
 Recall that it follows from Lemma \ref{lem31} that there exists $\omega\ge 0$ such that the semigroup $(e^{-\omega t}T(t))_{t\ge 0}$ has a fractional Gaussian estimate. Therefore, without loss of generality, we can assume that $\omega=0$, that is, $T$ has a fractional Gaussian estimate.
We prove the result in several steps.\\
 
{\bf Step 1}.  
Let $0<\varepsilon<1$, $0<\theta<\frac{\pi}{2}$ and $0<\theta_1<\varepsilon\theta$ be fixed. Using  \eqref{estma} and \eqref{in-rep}, we get that there is a constant $C>0$ such that 
 
\begin{align}\label{EST-A}
\abs{\left(T(z)f\right)(x)}&\leq C\Big (\operatorname{Re} (z)\Big)^{-\frac{N}{2s}}\int_{\Omega}\left(1+\abs{x-y}\abs{bz}^{-1/2s}\right)^{-(N+2s)(1-\varepsilon)}\abs{f(y)}\;dy\notag\\
&\leq  C\Big(\operatorname{Re} (z)\Big)^{-\frac{N}{2s}}\int_{\R^N}\left(1+\abs{x-y}\abs{bz}^{-1/2s}\right)^{-(N+2s)(1-\varepsilon)}\abs{\widetilde{f}(y)}\;dy,
\end{align}
where $\widetilde{f}$ is the extension of $f$ by zero on $\RR^N\setminus\Omega$.
Using a change of variable and  Young's convolution inequality we get from \eqref{EST-A} that
\begin{align}\label{Int-es}
\norm{T(z)f}_{L^p(\Omega)}&\leq C\Big(\operatorname{Re} (z)\Big)^{-\frac{N}{2s}}\norm{f}_{L^p(\Omega)}\int_{\R^N}\left(1+\abs{x}\abs{z}^{-1/2s}\right)^{-(N+2s)(1-\varepsilon)}dx\notag\\
&\le C\left(\frac{\abs{z}}{\operatorname{Re} (z)}\right)^{\frac{N}{2s}}\norm{f}_{L^p(\Omega)}\int_{\R^N}\left(1+\abs{x}\right)^{-(N+2s)(1-\varepsilon)}\;dx.
\end{align}

Taking $0<\varepsilon<\frac{2s}{N+2s}<1$, we get that
\begin{align}\label{Int-f}
\int_{\R^N}\left(1+\abs{x}\right)^{-(N+2s)(1-\varepsilon)}\;dx<\infty.
\end{align}
If follows from \eqref{Int-es} and \eqref{Int-f} that there is a constant $C>0$ such that

\begin{equation}
  \norm{T(z)}_{\mathcal{L}(L^p(\Omega))}\leq C \left(\frac{1}{\cos \theta_1}\right)^{\frac{N}{2s}}
  \label{bound1}
\end{equation}
for all $z\in\Sigma(\theta_1)$ and $1\leq p<\infty$.

We have shown that $T_p(z)\in\mathcal{L}(L^p (\Omega))$ and $T_p(z)f=T(z)f$ for every  $f\in L^p(\Omega)\cap L^2(\Omega)$, $z\in\Sigma(\theta_1)$ and $1\leq p<\infty$.   Since $L^2(\Omega)\cap L^p(\Omega)$ is dense in $L^p(\Omega)$, it follows that $T_p(z_1+z_2)=T_p(z_1)T_p(z_2)$ for all $z_1,z_2\in\Sigma(\theta_1)$. In addition we have that the norm $\norm{T_p(z)}_{\mathcal{L}(L^p(\Omega))}$ is bounded in $\Sigma(\theta_1)$.\\

{\bf Step 2}. We claim  that $T_p:\Sigma(\theta_1)\rightarrow \mathcal{L}(L^p(\Omega))$ is holomorphic.  Recall from \cite[Appendix A]{ABHN} that this  is equivalent to weak holomorphy. In other words, we have to show that the mapping $z\mapsto  \langle T_p(z)f,g\rangle$ is holomorphic for each $f\in L^p(\Omega)$ and $g\in L^q(\Omega)$ with $\frac 1p+\frac 1q=1$, where $\langle\cdot,\cdot\rangle$ denotes the duality pairing between $L^p(\Omega)$ and $L^q(\Omega)$. We prove the case $p=1$ (the case $p\neq 1$ being similar). Indeed
let $f\in L^1(\Omega)$ and $g\in L^{\infty}(\Omega)$. Let $\{f_n\}_{n\in \NN}\subset L^1(\Omega)\cap L^2(\Omega)$ be a sequence such that $f_n\rightarrow f$ in $L^1(\Omega)$ as $n\to\infty$.  Let  $g_n:=\chi_{\Omega_n}g$ where $\{\Omega_n\}_{n\in\NN}$ is a sequence of bounded open sets satisfying $\bigcup_{n=1}^\infty\Omega_n=\Omega$.

Since $f_n, g_n\in L^2(\Omega)$ for each $n\geq 1$, we have that $\langle T_1(z)f_n,g_n\rangle=( T(z) f_n,g_n)_{L^2(\Omega)}$.
Hence,  $\langle T_1(z)f_n,g_n\rangle$ is holomorphic. Using (\ref{bound1})we get that there is a constant $C>0$ such that

\begin{align*}
\abs{\langle T_1(z)f_n,g_n\rangle}\le \|T_1(z)f_n\|_{L^1(\Om)}\|g_n\|_{L^\infty(\Om)}\leq C \left(\frac{1}{\cos \theta_1}\right)^{\frac{N}{2s}}\norm{f_n}_{L^1(\Omega)}\norm{g_n}_{L^{\infty}(\Omega)}.
\end{align*}

Notice that $\|g_n\|_{L^\infty(\Omega)}\le \|g\|_{L^\infty(\Omega)}$ and since the sequence $\{f_n\}_{n\in\NN}$ converges, then it is bounded.
Therefore, there is a constant $C>0$ such that for all $n\geq 1$ and $z\in\Sigma(\theta_1)$ we have

\begin{equation}
\abs{\langle T_1(z)f_n,g_n\rangle}\leq C \left(\frac{1}{\cos \theta_1}\right)^{\frac{N}{2s}}\norm{g}_{L^{\infty}(\Omega)}.
\label{bound2}
\end{equation}
In addition we have that $\langle T_1(z)f_n,g_n\rangle\rightarrow \langle T_1(z)f,g\rangle$ as $n\rightarrow\infty$. It follows from \eqref{bound2} and Vitali's theorem (see \cite[Theorem A.5]{ABHN}) that $\langle T_1(z)f,g\rangle$ is holomorphic in $\Sigma(\theta_1)$ and the claim is proved. \\

{\bf Step 3}. We claim that $T_1(z)$ is strongly continuous.
That is, $T_1(z)f\rightarrow f$ as $z\rightarrow 0$, $z\in\Sigma(\theta_1)$ for each $f\in L^1(\Omega)$. It suffices to prove the claim for $f\in L^1(\Omega)\cap L^2(\Omega)$. Let $f\in L^1(\Omega)\cap L^2(\Omega)$.
Then $T_1(t)f=T(t)f$ for $t\geq0$ and $T(t)f\rightarrow f$ in $L^2(\Omega)$ as $t\downarrow 0$. 
Therefore,  there exists a sequence $t_n>0$ such that  $T_1(t_n)f(x)\rightarrow f(x)$ for a.e. $x\in\Omega$ as $t_n\downarrow 0$. Using  \eqref{Levy} 
and the Lebesgue Dominated Convergence Theorem, we can deduce that $T_1(t_n)f\rightarrow f$ in $L^1(\Omega)$ as $t_n\downarrow 0$. Since the mapping $z\mapsto T_1(z)$ is holomorphic, it follows that  $ T_1(z)f\rightarrow f$ in $L^1(\Omega)$ as $z\rightarrow0$, $z\in\Sigma(\theta_1)$ and we have proved the claim.\\

{\bf Step 4}. We have shown that the semigroup $T_1$ is bounded holomorphic of angle $\theta_1$. Since  $0<\theta_1< \frac{\pi}{2}$ was arbitrary, this implies that the semigroup is bounded holomorphic of angle $\frac{\pi}{2}$.  The proof is finished.
\end{proof}

%\subsection{Proof of Corollary \ref{Theo-C} } 

The proof of Corollary \ref{Theo-C} follows the lines of the case $s=1$ (the classical Gaussian estimate) contained in \cite[Corollary 2.5]{Ouhabaz}. For the sake of completeness we include the full proof.

\begin{proof}[\bf Proof of Corollary \ref{Theo-C}]
We give the proof for the case $C_0(\Om)$ (the case of  $C(\bOm)$ being similar).
As in the proof of Thoerem \ref{Theo-L11}, without any restriction we may assume that $T$ has a fractional Gaussian estimate for all $t\ge 0$. Denote by $A_1$ and $A_0$ the generators of the semigroups $T_1$ and $T_0$, respectively, and let $A_\infty=A_1^\star$ on $L^\infty(\Omega)$.

Firstly, we show that $\Sigma(\pi)\subset\rho(A_0)$. Since $T_1$ is holomorphic of angle $\frac{\pi}{2}$, then $\Sigma({\pi})\subset\rho(A_1)=\rho(A_\infty)$.  Let $\lambda\in\partial\sigma(A_0)$ (the boundary of the spectrum $\sigma(A_0)$). Then $\lambda$ is in the approximate point spectrum of $A_0$. This means that there is a sequence $\{f_n\}_{n\in\NN}\subset D(A_0)$ such that $\|f_n\|_{C_0(\Omega)}=1$ and $(\lambda-A_0)f_n\to 0$ in $C_0(\Omega)$ as $n\to\infty$ (see e.g. \cite[Page 64]{Nagel-al}). On the other hand, it is readily seen that $A_\infty$ is an extension of $A_0$ and thus $\lambda\in \partial\sigma(A_\infty)=\sigma(A_\infty)\subset (-\infty,0]$. We have shown that $\sigma(A_0)=\partial\sigma(A_0)\subset \sigma(A_\infty)$ and this implies that $\Sigma(\pi)\subset\rho(A_\infty)\subset\rho(A_0)$.

Secondly, since $(\lambda-A_0)^{-1}=(\lambda-A_\infty)^{-1}$ on $C_0(\Omega)$ for every  $\lambda>0$, we have that the equality also holds for every $\lambda\in \Sigma(\pi)$ by analytic continuation.  Then, by Remark \ref{rem2}, there is a constant $C>0$ such that

\begin{align*}
\|(\lambda-A_0)^{-1}\|_{\mathcal L(C_0(\Omega))}\le \|(\lambda-A_\infty)^{-1}\|_{\mathcal L(L^\infty(\Omega))}=\|(\lambda-A_1)^{-1}\|_{\mathcal L(L^1(\Omega))}\le \frac{C}{|\lambda|}
\end{align*}
for every $\lambda\in \Sigma(\pi)$. By Remark \ref{rem2} again, the preceding estimate implies that $T_0$ is bounded holomorphic of angle $\frac{\pi}{2}$ on $C_0(\Omega)$. The proof is finished.
\end{proof}

\bibliographystyle{plain}
\bibliography{biblio}

\end{document}